\documentclass[a4paper,UKenglish]{llncs}

\usepackage{microtype}
\usepackage{comment}
\usepackage{stmaryrd}
\usepackage{epstopdf}
\usepackage{thmtools}

\usepackage{xcolor}
\usepackage{amsmath}
\usepackage{amssymb}
\usepackage{float}
\usepackage{enumerate}
\usepackage{mathrsfs}
\usepackage{wasysym}
 \usepackage{graphicx}

\def\NN{\mathbb{N}}
\def\QQ{\mathbb{Q}}
\def\ZZ{\mathbb{Z}}
\def\ZZd{\mathbb{Z}^d}

\def\ag{\mathcal{A}}
\def\bg{\mathcal{B}}

\def\FF{\mathcal{F}}

\def\EE{\mathbb{E}}
\newcommand{\End}{\mathrm{End}}
\newcommand{\Aut}{\mathrm{Aut}}

\newcommand{\TM}{\mathrm{TM}}
\newcommand{\RTM}{\mathrm{RTM}}
\newcommand{\OB}{\mathrm{OB}}
\newcommand{\LP}{\mathrm{LP}}
\newcommand{\FG}{\mathrm{RFA}}
\newcommand{\EL}{\mathrm{EL}}
\newcommand{\SHIFT}{\mathrm{Shift}}
\newcommand{\SP}{\mathrm{SP}}
\newcommand{\define}[1]{\emph{#1}}

\newcommand{\ID}{\mathrm{id}}
\newcommand{\fix}{\mathrm{fix}}

\newcommand{\Sym}{\mathrm{Sym}}
\newcommand{\Alt}{\mathrm{Alt}}

\usepackage{tikz}
\usetikzlibrary{patterns,positioning,arrows,decorations.markings,calc,decorations.pathmorphing,decorations.pathreplacing}
\usetikzlibrary{lindenmayersystems}
\usetikzlibrary{arrows}
\usepackage{tikz-3dplot}

\definecolor{vert}{RGB}{0,178,102}

\title{The group of reversible Turing machines}

\author{Sebasti\'an Barbieri\inst{1}, Jarkko Kari\inst{2} \and Ville Salo\inst{3}}

\institute{LIP, ENS de Lyon -- CNRS -- INRIA -- UCBL -- Universit\'e de Lyon,\\
	\and
	University of Turku\\
	\and
	Center for Mathematical Modeling, University of Chile \\
	\email{\texttt{sebastian.barbieri@ens-lyon.fr}}, \email{\texttt{jkari@utu.fi}}, \email{\texttt{vosalo@utu.fi}}
}
\authorrunning{S. Barbieri, J. Kari and V. Salo} 

\begin{document}

	\maketitle 
	\begin{abstract}
		We consider Turing machines as actions over configurations in $\Sigma^{\ZZ^d}$ which only change them locally around a marked position that can move and carry a particular state. In this setting we study the monoid of Turing machines and the group of reversible Turing machines. We also study two natural subgroups, namely the group of finite-state automata, which generalizes the topological full groups studied in the theory of orbit-equivalence, and the group of oblivious Turing machines whose movement is independent of tape contents, which generalizes lamplighter groups and has connections to the study of universal reversible logical gates. Our main results are that the group of Turing machines in one dimension is neither amenable nor residually finite, but is locally embeddable in finite groups, and that the torsion problem is decidable for finite-state automata in dimension one, but not in dimension two.
	\end{abstract}

	\section{Introduction}	\label{section.introduction}
	
	\subsection{Turing machines and their generalization}
	
	
	Turing machines have been studied since the 30s as the standard formalization of the abstract concept of computation. However, more recently, Turing machines have also been studied in the context of dynamical systems. In \cite{Ku97}, two dynamical systems were associated to a Turing machine, one with a `moving tape' and one with a `moving head'. After that, there has been a lot of study of dynamics of Turing machines, see for example \cite{DeKuBl06,BlCaNi02,KaOl08,GaMa07,GaGu10,Je13,GaOlTo15}. Another connection between Turing machines and dynamics is that they can be used to define subshifts. Subshifts whose forbidden patterns are given by a Turing machine are called effectively closed, computable, or $\Pi^0_1$ subshifts, and especially in multiple dimensions, they are central to the topic due to the strong links known between SFTs, sofic shifts and $\Pi^0_1$-subshifts, see for example \cite{DuRoSh10,AuSa13}. An intrinsic notion of Turing machine computation for these subshifts on general groups was proposed in \cite{AuBaSa14}, and a similar study was performed with finite state machines in \cite{SaTo15,SaTo15c}.
	
	In all these papers, the definition of a Turing machine is (up to notational differences and switching between the moving tape and moving head model) the following: A Turing machine is a function $T : \Sigma^\ZZ \times Q \to \Sigma^\ZZ \times Q$ defined by a local rule $f_T : \Sigma \times Q \to \Sigma \times Q \times \{-1,0,1\}$ by the formula
\[ T(x, q) = (\sigma_{-d}(\tilde{x}), q') \mbox{ if } f_T(x_0, q) = (a,q',d), \]
where $\sigma: \Sigma^{\ZZ} \to \Sigma^{\ZZ}$ is the shift action given by $\sigma_{d}(x)_{z} = x_{z-d}$, $\tilde{x}_0 = a$ and $\tilde{x}|_{\ZZ \setminus \{0\}} = x|_{\ZZ \setminus \{0\}}$. In this paper, such Turing machines are called \emph{classical Turing machines}. This definition (as far as we know) certainly suffices to capture all computational and dynamical properties of interest, but it also has some undesirable properties: The composition of two classical Turing machines -- and even the square of a classical Turing machine -- is typically not a classical Turing machine, and the reverse of a reversible classical Turing machine is not always a classical Turing machine.

	In this paper, we give a more general definition of a Turing machine, by allowing it to move the head and modify cells at an arbitrary (but bounded) distance on each timestep. With the new definition, we get rid of both issues: With our definition,
\begin{itemize}
\item Turing machines are closed under composition, forming a monoid, and 
\item reversible Turing machines are closed under inversion, forming a group.
\end{itemize}
We also characterize reversibility of classical Turing machines in combinatorial terms, and show what their inverses look like. Our definition of a Turing machine originated in the yet unpublished work \cite{SaScworkinprogress}, where the group of such machines was studied on general $\ZZ$-subshifts (with somewhat different objectives).

	These benefits of the definition should be compared to the benefits of allowing arbitrary radii in the definition of a cellular automaton: If we define cellular automata as having a fixed radius of, say, $3$, then the inverse map of a reversible cellular automaton is not always a cellular automaton, as the inverse of a cellular automaton may have a much larger radius \cite{CzKa07}. Similarly, with a fixed radius, the composition of two cellular automata is not necessarily a cellular automaton.
	
	We give our Turing machine definitions in two ways, with a moving tape and with a moving head, as done in \cite{Ku97}. The moving tape point of view is often the more useful one when studying one-step behavior and invariant measures, whereas we find the moving head point of view easier for constructing examples, and when we need to track the movement of multiple heads. The moving head Turing machines are in fact a subset of cellular automata on a particular kind of subshift. The moving tape machine on the other hand is a generalization of the topological full group of a subshift, which is an important concept in particular in the theory of orbit equivalence. For topological full groups of minimal subshifts, see for example \cite{GiPuSk99,GrMe11,JuMo12}. The (one-sided) SFT case is studied in \cite{Ma12}.
	

	\subsection{Our results and comparisons with other groups}
	
	In Section~\ref{section.preliminaries}, we define our basic models and prove basic results about them. In Section~\ref{subsection_measure}, we define the uniform measure and show as a simple application of it that injectivity and surjectivity are both equal to reversibility.
	
	Our results have interesting counterparts in the theory of cellular automata: One of the main theorems in the theory of cellular automata is that injectivity implies surjectivity, and (global) bijectivity is equivalent to having a cellular automaton inverse map. Furthermore, one can attach to a reversible one- or two-dimensional cellular automaton its `average drift', that is, the speed at which information moves when the map is applied, and this is a homomorphism from the group of cellular automata to a sublattice of $\QQ^d$ (where $d$ is the corresponding dimension), see \cite{Ka96}. In Section~\ref{section.subgroups} we use the uniform measure to define an analog, the `average movement' homomorphism for Turing machines.
	
	
	In Section~\ref{section.subgroups}, we define some interesting subgroups of the group of Turing machines. First, we define the local permutations -- Turing machines that never move the head at all --, and their generalization to oblivious Turing machines where movement is allowed, but is independent of the tape contents. The group of oblivious Turing machines can be seen as a kind of generalization of lamplighter groups. It is easy to show that these groups are amenable but not residually finite. What makes them interesting is that the group of oblivious Turing machines is finitely generated, due to the existence of universal reversible logical gates. It turns out that strong enough universality for reversible gates was proved only recently \cite{AaGrSc15}. 
	
	We also define the group of (reversible) finite-state machines -- Turing machines that never modify the tape. Here, we show how to embed a free group with a similar technique as used in \cite{ElMo12}, proving that this group is non-amenable. By considering the action of Turing machines on periodic points,\footnote{The idea is similar as that in \cite{BoLiRu88} for showing that automorphism groups of mixing SFTs are residually finite, but we do not actually look at subsystems, but the periodic points of an enlarged system, where we allow infinitely many heads to occur.} we show that the group of finite-state automata is residually finite, and the group of Turing machines is locally embeddable in finite groups (in particular sofic).
	
	Our definition of a Turing machine can be seen as a generalization of the topological full group, and in particular finite-state machines with a single state exactly correspond to this group. Thus, it is interesting to compare the results of Section~\ref{section.subgroups} to known results about topological full groups. In \cite{GrMe11,JuMo12} it is shown that the topological full group of a minimal subshift is locally embeddable in finite groups and amenable, while we show that on full shifts, this group is non-amenable, but the whole group of Turing machines is LEF.\footnote{In \cite{SaScworkinprogress} it is shown that on minimal subshifts, the group of Turing machines coincides with the group of finite-state automata.}
	
	Our original motivation for defining these subgroups -- finite-state machines and local permutations -- was to study the question of whether they generate all reversible Turing machines. Namely, a reversible Turing machine changes the tape contents at the position of the head and then moves, in a globally reversible way. Thus, it is a natural question whether every reversible Turing machine can actually be split into reversible tape changes (actions by local permutations) and reversible moves (finite-state automata). We show that this is not the case, by showing that Turing machines can have arbitrarily small average movement, but that elementary ones have only a discrete sublattice of possible average movements. We do not know whether this is the only restriction.
	
	
	In Section~\ref{section.computability}, we show that the group of Turing machines is recursively presented and has a decidable word problem, but that its torsion problem (the problem of deciding if a given element has finite order) is undecidable in all dimensions. For finite-state machines, we show that the torsion problem is decidable in dimension one, but is undecidable in higher dimensions, even when we restrict to a finitely generated subgroup. We note a similar situation with Thompson's group $V$: its torsion problem is decidable in one-dimension, but undecidable in higher dimensions. \cite{BeMa14,BeBl14}

	%
	
	\subsection{Preliminaries}
	In this section we present general definitions and settle the notation which is used throughout the article. The review of these concepts will be brief and focused on the dynamical aspects. For a more complete introduction the reader may refer to~\cite{LiMa95} or~\cite{CeCo10} for the group theoretic aspects. Let $\ag$ be a finite alphabet. The set $\ag^{\ZZd} = \{ x: \ZZd \to \ag\}$ equipped with the left group action $\sigma: \ZZd \times \ag^{\ZZd} \to \ag^{\ZZd}$ defined by $(\sigma_{\vec{v}}(x))_{\vec{u}} = x_{\vec{u}-\vec{v}}$ is a \textit{full shift}. The elements $a \in \ag$ and $x \in \ag^{\ZZd}$ are called \define{symbols} and \define{configurations} respectively. With the discrete topology on $\ag$ the set of configurations $\ag^{\ZZd}$ is compact and given by the metric $d(x,y) = 2^{-\inf (\{ |\vec{v}| \in \NN \mid\ x_{\vec{v}} \neq y_{\vec{v}}  \})}$ where $|\vec{v}|$ is a norm on $\ZZ^d$ (we settle here for the $||\cdot||_{\infty}$ norm). This topology has the \define{cylinders} $[a]_{\vec{v}} = \{x \in \ag^{\ZZd} | x_{\vec{v}} = a\in \ag\}$ as a subbasis. A \emph{support} is a finite subset $F \subset \ZZd$. Given a support $F$, a \emph{pattern with support $F$} is an element $p$ of $\ag^F$. We also denote the cylinder generated by $p$ in position $\vec{v}$ as $[p]_{\vec{v}} = \bigcap_{\vec{u} \in F} [p_{\vec{u}}]_{\vec{v}+\vec{u}}$, and $[p]= [p]_{\vec{0}}$.
	\begin{definition}
		A subset $X$ of $\ag^{\ZZd}$ is a \define{subshift} if it is $\sigma$-invariant -- $\sigma(X)\subset X$ -- and closed for the cylinder topology. Equivalently, $X$ is a subshift if and only if there exists a set of forbidden patterns $\FF$ that defines it.
		$$X = \bigcap_{p \in \FF, \vec{v} \in \ZZd}{\ag^{\ZZd} \setminus [p]_{\vec{v}}}.$$
	\end{definition}
	
	Let $X,Y$ be subshifts over alphabets $\ag$ and $\bg$ respectively. A continuous $\ZZd$-equivariant map $\phi : X \to Y$ between subshifts is called a morphism. A well-known Theorem of Curtis, Lyndon and Hedlund which can be found in full generality in~\cite{CeCo10} asserts that morphisms are equivalent to maps defined by local rules as follows: There exists a finite $F \subset \ZZd$ and $\Phi: \ag^{F} \to \bg$ such that $\forall x \in X: \phi(x)_{\vec{v}} = \Phi(\sigma_{-\vec{v}}(x)|_F)$. If $\phi$ is an endomorphism then we refer to it as a cellular automaton. A cellular automaton is said to be reversible if there exists a cellular automaton $\phi^{-1}$ such that $\phi \circ \phi^{-1} = \phi^{-1} \circ \phi = \ID$. It is well known that reversibility is equivalent to bijectivity.
	
	Throughout this article we use the following notation inspired by Turing machines. We denote by $\Sigma = \{0,\dots,n-1\}$ the set of tape symbols and $Q = \{1,\dots,k\}$ the set of states. We also use exclusively the symbols $n = |\Sigma|$ for the size of the alphabet and $k = |Q|$ for the number of states. Given a function $f : \Omega \to \prod_{i \in I}A_i$ we denote by $f_i$ the projection of $f$ to the $i$-th coordinate. 
	
	\section{Two models for Turing machine groups}\label{section.preliminaries}
	
	In this section we define our generalized Turing machine model, and the group of Turing machines. In fact, we give two definitions for this group, one with a moving head and one with a moving tape as in \cite{Ku97}. We show that -- except in the case of a trivial alphabet -- these groups are isomorphic.\footnote{Note that the \emph{dynamics} obtained from these two definitions are in fact quite different, as shown in \cite{Ku97,Ku10}.} Furthermore, both can be defined both by local rules and `dynamically', that is, in terms of continuity and the shift. In the moving tape model we characterize reversibility as preservation of the uniform measure. Finally we conclude this section by characterizing reversibility for classical Turing machines in our setting.

	\subsection{The moving head model}
	
	Consider $Q = \{1,\dots,k \}$ and let $X_k$ be the subshift with alphabet $Q \cup \{0\}$ such that in each configuration the number of non-zero symbols is at most one. \[ X_k = \{ x \in \{0,1,\ldots,k\}^{\ZZd} \ |\  0 \notin \{x_{\vec{u}}, x_{\vec{v}}\} \implies \vec{u} = \vec{v} \}. \] In particular $X_0 = \{0^{\ZZd}\}$ and $i < j \implies X_i \subsetneq X_j$. Let also $\Sigma = \{0,\dots, n-1\}$ and $X_{n,k} = \Sigma^{\ZZ^d} \times X_{k}$. For the case $d = 1$, configurations in $X_{n,k}$ represent a bi-infinite tape filled with symbols in $\Sigma$ possibly containing a head that has a state in $Q$. Note that there might be no head in a configuration.
	
	\begin{definition}
		Given a function
		\[ f : \Sigma^{F} \times Q \to \Sigma^{F'} \times Q \times \ZZd, \]
		where $F,F'$ are finite subsets of $\ZZd$, we can define a map $T_f : X_{n,k} \to X_{n,k}$ as follows: Let $(x,y) \in X_{n,k}$. If there is no ${\vec v} \in \ZZd$ such that $y_{\vec v} \neq 0$ then $T(x,y) = (x,y)$. Otherwise let $p = \sigma_{-{\vec v}}(x)|_{F}$, $q = y_{\vec v} \neq 0$ and $f(p,q) = (p',q',{\vec d}).$ Then $T(x,y) = (\tilde{x},\tilde{y})$ where:
		$$\tilde{x}_{\vec t} = \begin{cases}
		x_{\vec{t}} & \text{ if } {\vec t}-{\vec v} \notin F' \\ 
		p'_{{\vec t}-{\vec v}}&\text{ if } {\vec t}-{\vec v} \in F'
		\end{cases},\ \ \ \  \tilde{y}_{\vec t} = \begin{cases}
		q' & \text{ if } {\vec t} = {\vec v}+{\vec d}  \\ 
		0&\text{ otherwise }
		\end{cases}$$
		
		Such $T = T_f$ is called a \emph{(moving head) $(d,n,k)$-Turing machine}, and $f$ is its \emph{local rule}. If there exists a $(d,n,k)$-Turing machine $T^{-1}$ such that $T \circ T^{-1} = T^{-1} \circ T = \ID$, we say $T$ is \emph{reversible}.
	\end{definition}
	
	Note that $\sigma_{-{\vec v}}(x)|_{F}$ is the $F$-shaped pattern `at' $\vec v$, but we do not write $x|_{F + {\vec v}}$ because we want the pattern we read from $x$ to have $F$ as its domain.
	
	This definition corresponds to classical Turing machines with the moving head model when $d = 1$, $F = F' = \{0\}$ and $f(x,q)_3 \in \{-1,0,1\}$ for all $x, q$. By possibly changing the local rule $f$, we can always choose $F = [-r_i, r_i]^d$ and $F' = [-r_o,r_o]^d$ for some $r_i, r_o \in \NN$, without changing the Turing machine $T_f$ it defines. The minimal such $r_i$ is called the \emph{in-radius} of $T$, and the minimal $r_o$ is called the \emph{out-radius} of $T$. We say the in-radius of a Turing machine is $-1$ if there is no dependence on input, that is, the neighborhood $[-r_i, r_i]$ can be replaced by the empty set. Since $\Sigma^{F} \times Q$ is finite, the third component of $f(p,q)$ takes only finitely many values $\vec v \in \ZZd$. The maximum of $|\vec v|$ for such $\vec v$ is called the \emph{move-radius} of $T$. Finally, the maximum of all these three radii is called the \emph{radius} of $T$. In this terminology, classical Turing machines are those with in- and out-radius $0$, and move-radius $1$.

	\begin{definition}
	Define $\TM(\ZZd,n,k)$ as the set of $(d,n,k)$-Turing machines and $\RTM(\ZZd,n,k)$ the set of reversible $(d,n,k)$-Turing machines.
	\end{definition}
	
	In some parts of this article we just consider $d = 1$. In this case we simplify the notation and just write $\RTM(n,k) := \RTM(\ZZ,n,k)$. 
	
	Of course, we want $\TM(\ZZd,n,k)$ to be a monoid and $\RTM(\ZZd,n,k)$ a group under function composition. This is indeed the case, and one can prove this directly by constructing local rules for the inverse of a reversible Turing machine and composition of two Turing machines. However, it is much easier to extract this from the following characterization of Turing machines as a particular kind of cellular automaton.
	
	For a subshift $X$, we denote by $\End(X)$ the monoid of endomorphisms of $X$ and $\Aut(X)$ the group of automorphisms of $X$. 
	

	\begin{restatable}{proposition}{propCACharacterization}
	\label{prop_CACharacterization}
	Let $n,k$ be positive integers and $Y = X_{n,0}$. Then:
	\begin{align*}
	\TM(\ZZd,n,k) &= \{ \phi \in \End(X_{n,k}) \mid \phi|_Y = \ID, \phi^{-1}(Y) = Y \} \\
	\RTM(\ZZd,n,k) &= \{ \phi \in \Aut(X_{n,k}) \mid \phi|_Y = \ID  \}
	\end{align*}
	\end{restatable}
	 
	 \begin{corollary}
	 We have $\phi \in \RTM(\ZZd,n,k)$ if and only if $\phi \in \TM(\ZZd,n,k)$ and $\phi$ is bijective.
	 \end{corollary}
	 
	 Clearly, the conditions of Proposition~\ref{prop_CACharacterization} are preserved under function composition and inversion. Thus:
	 
	 \begin{corollary}
	 Under function composition, $(\TM(\ZZd,n,k), \circ)$ is a submonoid of $\End(X_{n,k})$ and $(\RTM(\ZZd,n,k), \circ)$ is a group.
	 \end{corollary}
	 
	 We usually omit the function composition symbol, and use the notations $\TM(\ZZd,n,k)$ and $\RTM(\ZZd,n,k)$ to refer to the corresponding monoids and groups.
	 
	\subsection{The moving tape model}
	
	It's also possible to consider the position of the Turing machine as fixed at $0$, and move the tape instead, to obtain the moving tape Turing machine model. In \cite{Ku97}, where Turing machines are studied as dynamical systems, the moving head model and moving tape model give non-conjugate dynamical systems. However, the abstract monoids defined by the two points of view turn out to be equal, and we obtain an equivalent definition of the group of Turing machines. 
	
	
	As in the previous section, we begin with a definition using local rules.
	
	\begin{definition}
	Given a function
	$f : \Sigma^{F} \times Q \to \Sigma^{F'} \times Q \times \ZZd$,
	where $F,F'$ are finite subsets of $\ZZd$, we can define a map $T_{f} : \Sigma^{\ZZd} \times Q \to \Sigma^{\ZZd} \times Q$ as follows: If $f(x|_F, q) = (p, q', \vec d)$, then $T_f(x, q) = (\sigma_{\vec d}(y), q')$
	where
	\[ y_{\vec u} = \left\{\begin{array}{cc} 
	x_{\vec u}, &\mbox{if } \vec u \notin F' \\
	p_{\vec u}, &\mbox{if } \vec u \in F',
	\end{array}\right. \]
	is called the \emph{moving tape Turing machine} defined by $f$.
	\end{definition}
	
	These machines also have the following characterization with a slightly more dynamical feel to it. Say that $x$ and $y$ are \emph{asymptotic}, and write $x \sim y$, if $d(\sigma_{\vec v}(x), \sigma_{\vec v}(y)) \rightarrow 0$ as $|\vec v| \rightarrow \infty$. 
	We write $x \sim_m y$ if $x_{\vec v} = y_{\vec v}$ for all $|\vec v| \geq m$, and clearly $x \sim y \iff \exists m: x \sim_m y$.
	
	\begin{restatable}{lemma}{MovingTapeDynamicalDef}
	\label{lem:MovingTapeDynamicalDef}
	Let $T : \Sigma^{\ZZd} \times Q \to \Sigma^{\ZZd} \times Q$ be a function. Then $T$ is a moving tape Turing machine if and only if it is continuous, and for a continuous function $s : \Sigma^{\ZZd} \times Q \to \ZZd$ and $a \in \NN$ we have $T(x, q)_1 \sim_a \sigma_{s(x, q)}(x)$ for all $(x, q) \in \Sigma^{\ZZd} \times Q$.
	\end{restatable}

	Note that in place of $a$ we could allow a continuous $\NN$-valued function of $(x, q)$ -- the definition obtained would be equivalent, as the $a$ of the present definition can be taken as the maximum of such a function.
	
	We call the function $s$ in the definition of these machines the \define{shift indicator} of $T$, as it indicates how much the tape is shifted depending on the local configuration around $0$. In the theory of orbit equivalence and topological full groups, the analogs of $s$ are usually called \emph{cocycles}. We also define in-, out- and move-radii of moving tape Turing machines similarly as in the moving head case.
	
	We note that it is not enough that $T(x, q)_1 \sim \sigma_{s(x, q)}(x)$ for all $(x, q) \in \Sigma^{\ZZd} \times Q$: Let $Q = \{1\}$ and consider the function $T : \Sigma^{\ZZ} \times Q \to \Sigma^{\ZZ} \times Q$ defined by $(T(x, 1)_1)_i = x_{-i}$ if $x_{[-|i|+1,|i|-1]} = 0^{2i-1}$ and $\{x_i,x_{-i}\} \neq \{0\}$, and $(T(x, 1)_1)_i = x_i$ otherwise. Clearly this map is continuous, and the constant-$\vec 0$ map $s(x,q) = \vec 0$ gives a shift-indicator for it. However, $T$ is not defined by any local rule since it can modify the tape arbitrarily far from the origin.
	
%
As for moving head machines, it is easy to see (either by constructing local rules or by applying the dynamical definition) that the composition of two moving tape Turing machines is again a moving tape Turing machine. This allows us to proceed as before and define their monoid and group.

	\begin{definition}
		We denote by $\TM_{\fix}(\ZZd,n,k)$ and $\RTM_{\fix}(\ZZd,n,k)$ the monoid of moving tape $(d,n,k)$-Turing machines and the group of reversible moving tape $(d,n,k)$-Turing machines respectively.
	\end{definition}
	
	Now, let us show that the moving head and moving tape models are equivalent. First, there is a natural epimorphism $\Psi: \TM(\ZZd,n,k) \to \TM_{\fix}(\ZZd,n,k)$. Namely, let $T \in \TM(\ZZd,n,k)$. We define $\Psi(T)$ as follows: Let $(x,q) \in \Sigma^{\ZZd}\times Q$. Letting $y$ be the configuration such that $y_{\vec 0} = q$ and $0$ everywhere else and $T(x,y) = (x',y')$ such that $y'_{\vec v} = q'$ we define $\Psi(T)(x,q) = (\sigma_{-\vec v}(x'),q')$. This is clearly an epimorphism but it's not necessarily injective if $n = 1$. Indeed, we have that $\RTM_{\fix}(\ZZd,1,k) \cong S_k$ and $\TM_{\fix}(\ZZd,1,k)$ is isomorphic to the monoid of all functions from $\{1,\dots,k\}$ to itself while $\ZZ^d \leq \RTM(\ZZd,1,k) \leq \TM(\ZZd,1,k)$. Nevertheless, if $n \geq 2$ this mapping is an isomorphism.
		
	\begin{restatable}{lemma}{GroupIsomorphismLemma}
	\label{group_isomorphism_lemma}
		If $n \geq 2$ then:
		\begin{align*}
		\TM_{\fix}(\ZZd,n,k) &\cong \TM(\ZZd,n,k) \\
		\RTM_{\fix}(\ZZd,n,k) &\cong \RTM(\ZZd,n,k).
		\end{align*}
	\end{restatable} 

	The previous result means that besides the trivial case $n = 1$ where the tape plays no role, we can study the properties of these groups using any model. 
	 
	
	\subsection{The uniform measure and reversibility.}\label{subsection_measure}
	
	Consider the space $\Sigma^{\ZZd} \times Q$. We define a measure $\mu$ on $\mathcal{B}(\Sigma^{\ZZd} \times Q)$
	 as the product measure of the uniform Bernoulli measure and the uniform discrete measure. That is, if $F$ is a finite subset of $\ZZd$ and $p \in \Sigma^F$, then: $$\mu([p]\times \{q\}) = \frac{1}{kn^{|F|}}.$$
	 
	
	With this measure in hand we can prove the following:
	
	\begin{restatable}{theorem}{InjectiveSurjectiveReversible}
	\label{theorem_injective_surjective_reversible}
		Let $T \in \TM_{\fix}(\ZZd,n,k)$. Then the following are equivalent:
		\begin{enumerate}
			\item $T$ is injective.
			\item $T$ is surjective.
			\item $T \in \RTM_{\fix}(\ZZd,n,k)$.
			\item $T$ preserves the uniform measure ($\mu(T^{-1}(A)) = \mu(A)$ for all Borel sets $A$).
			\item $\mu(T(A)) = \mu(A)$ for all Borel sets $A$.
		\end{enumerate}
	\end{restatable}
	
	\begin{remark}
	The proof is based on showing that every Turing machine is a \define{local homeomorphism} and preserves the measure of all large-radius cylinders in the forward sense. Note that preserving the measure of large-radius cylinders in the forward sense does not imply preserving the measure of all Borel sets, in general. For example, the machine which turns the symbol in $F =\{\vec{0}\}$ to $0$ without moving the head satisfies $\mu([p]) = \mu(T[p])$ for any $p \in \Sigma^S$ with $S \supset F$. But $\mu([]) = 1$ and $\mu(T([]))=\mu([0]) =1/2$, where $[] = \{0,\ldots,n\}^\ZZ$ is the cylinder defined by the empty word. 
	\end{remark}
	
	Using the measure, one can define the average movement of a Turing machine.
	
	\begin{definition}
		Let $T \in \TM_{\fix}(\ZZd,n,k)$ with shift indicator function $s : \Sigma^{\ZZd} \times Q \to \ZZd$. 
		We define the \emph{average movement} $\alpha(T) \in \QQ^d$ as
		\[ \alpha(T) := \EE_{\mu}(s) = \int_{\Sigma^{\ZZd}\times Q} s(x,q) d\mu, \]
		where $\mu$ is the uniform measure defined in Subsection~\ref{subsection_measure}.
		For $T$ in $\TM(\ZZd,n,k)$ we define $\alpha$ as the application to its image under the canonical epimorphism $\Psi$, that is, $\alpha(T) := \alpha(\Psi(T))$. 
	\end{definition}
	
	We remark that this integral is actually a finite sum over the cylinders $p \in \Sigma^F$. Nonetheless, its expression as an expected value allows us to show the following: If $T_1,T_2 \in \RTM(\ZZd,n,k)$ then $\alpha(T_1 \circ T_2) = \alpha(T_1)+\alpha(T_2)$. Indeed, as reversibility implies measure-preservation, we have that
	\[ \EE_{\mu}(s_{T_1 \circ T_2}) = \EE_{\mu}(s_{T_1}\circ T_2+s_{T_2}) = \EE_{\mu}(s_{T_1})+ \EE_{\mu}(s_{T_2}). \]
	This means that $\alpha$ defines an homomorphism from $\RTM(\ZZd,n,k)$ to $\QQ^d$.

	\subsection{Classical Turing machines}
	
	As discussed in the introduction, we say a one-dimensional Turing machine is \emph{classical} if its in- and out-radii are $0$, and its move-radius is $1$. In this section, we characterize reversibility in classical Turing machines. If $T_0$ has in-, out- and move-radius $0$, that is, $T_0$ only performs a permutation of the set of pairs $(s,q) \in \Sigma \times Q$ at the position of the head, then we say $T_0$ is a \emph{state-symbol permutation}. If $T_1$ has in-radius $-1$, never modifies the tape, and only makes movements by vectors in $\{-1,0,1\}$, then $T_1$ is called a \emph{state-dependent shift}.\footnote{Note that these machines are slightly different than the groups $\SP(\ZZ,n,k)$ and $\SHIFT(\ZZ,n,k)$ introduced in Section~\ref{section.subgroups}, as the permutations in $\SP(\ZZ,n,k)$ do not modify the tape, and moves in $\SHIFT(\ZZ;n,k)$ cannot depend on the state.}
	
	\begin{restatable}{theorem}{ClassicalReversibility}
	\label{thm:ClassicalReversibility}
	A classical Turing machine $T$ is reversible if and only if it is of the form $T_1 \circ T_0$ where $T_0$ is a state-symbol permutation and $T_1$ is a state-dependent shift. 
	\end{restatable}
	
	It follows that the inverse of a reversible classical Turing machine is always of the form $T_0 \circ T_1$ where $T_0$ is a state-symbol permutation and $T_1$ is a state-dependent shift. In the terminology of Section~\ref{section.subgroups}, the theorem implies that all reversible classical Turing machines are elementary.

	\section{Properties of $\RTM$ and interesting subgroups} 
	\label{section.subgroups}
	In this section we study some properties of $\RTM$ by studying the subgroups it contains. 
We introduce $\LP$, the group of local permutations where the head does not move and $\FG$, the group of (reversible) finite-state automata which do not change the tape. These groups separately capture the dynamics of changing the tape and moving the head. We also define the group of oblivious Turing machines $\OB$ as an extension of $\LP$ where arbitrary tape-independent moves are allowed, and $\EL$ as the group of elementary Turing machines, which are compositions of finite-state automata and oblivious Turing machines. 

	First, we observe that $\alpha(\RTM(\ZZd,n,k))$ is not finitely generated, and thus:
	
	\begin{restatable}{theorem}{RTMIsNotFinitelyGenerated}
	\label{RTM_is_not_finitely_generated}
		For $n \geq 2$, the group $\RTM(\ZZd,n,k)$ is not finitely generated.
	\end{restatable}
	
	Although $\alpha$ is not a homomorphism on $\TM(\ZZd,n,k)$, using Theorem~\ref{theorem_injective_surjective_reversible}, we obtain that $\TM(\ZZd,n,k)$ cannot be finitely generated either.

	\subsection{Local permutations and oblivious Turing machines}
	
	For $\vec v \in \ZZd$, define the machine $T_{\vec v}$ which does not modifies the state or the tape, and moves the head by the vector $\vec v$ on each step. Denote the group of such machines by $\SHIFT(\ZZd,n,k)$. Clearly $\alpha : \SHIFT(\ZZd,n,k) \to \ZZd$ is a group isomorphism. Define also $\SP(\ZZd,n,k)$ as the \emph{state-permutations}: Turing machines that never move and only permute their state as a function of the tape.
	
	\begin{definition}
		We define the group $\LP(\ZZd,n,k)$ of \emph{local permutations} as the subgroup of reversible $(d,n,k)$-Turing machines whose shift-indicator is the constant-$\vec{0}$ function. Define also $\OB(\ZZd,n,k) = \langle \SHIFT(\ZZd,n,k), \LP(\ZZd,n,k) \rangle$, the group of \emph{oblivious Turing machines}. 
	\end{definition}
	
	In other words, $\LP(\ZZd,n,k)$ is the group of reversible machines that do not move the head, and $\OB(\ZZd,n,k)$ is the group of reversible Turing machines whose head movement is independent of the tape contents. 
	Note that in the definition of both groups, we allow changing the state as a function of the tape, and vice versa. Clearly $\SHIFT(\ZZd,n,k) \leq \OB(\ZZd,n,k)$ and $\SP(\ZZd,n,k) \leq \LP(\ZZd,n,k)$.
	
	\begin{restatable}{proposition}{LPNotResiduallyFinite}
	\label{proposition_LP_not_residually_finite}
		Let $S_{\infty}$ be the group of permutations of $\NN$ of finite support. Then for $n \geq 2$, $S_{\infty} \hookrightarrow \LP(\ZZd,n,k)$.
	\end{restatable}
	
	
	In particular, $\RTM(\ZZd,n,k)$ is not residually finite.
	By Cayley's theorem, Proposition~\ref{proposition_LP_not_residually_finite} also implies that $\RTM(\ZZd,n,k)$ contains all finite groups.
	
	\begin{restatable}{proposition}{OBAmenable}
	\label{proposition_OB_amenable}
	The group $\OB(\ZZd,n,k)$ is amenable.
	\end{restatable}

	Write $H \wr G$ for the restricted wreath product.
	
	\begin{restatable}{proposition}{LamplighterEmbedding}
	\label{prop:LamplighterEmbedding}
	If $G$ is a finite group and $n \geq 2$, then $G \wr \ZZ^d \hookrightarrow \OB(\ZZd,n,k)$.
	\end{restatable}
	
	The groups $G \wr \ZZ^d$ are sometimes called generalized lamplighter groups. In fact, $\OB(\ZZd,n,k)$ can in some sense be seen as a generalized generalized lamplighter group, since the subgroup of $\OB(\ZZd,n,k)$ generated by the local permutations $\LP(\ZZd,n,1)$ with radius $0$ and $\SHIFT(\ZZd,n,1)$ is isomorphic to $A \cong S_n \wr \ZZd$.
	
	Interestingly, just like the generalized lamplighter groups, we can show that the whole group $\OB(\ZZd,n,k)$ is finitely generated. 
	
	\begin{restatable}{theorem}{OBFG}
	\label{thm:OBFG}
		$\OB(\ZZd,n,k)$ is finitely generated.
	\end{restatable}
	

	\subsection{Finite-state automata}

	\begin{definition}
		We define the \emph{reversible finite-state automata} $\FG(\ZZd,n,k)$ as the group of reversible $(d,n,k)$-Turing machines that do not change the tape. That is, the local rules are of the form $f(p,q) = (p,q',z)$ for all entries $p \in \Sigma^F, q \in Q$. 
	\end{definition}
	
	This group is orthogonal to $\OB(\ZZd,n,k)$ in the following sense:
	
	\begin{restatable}{proposition}{Orthogonal}
	\begin{align*}
	\FG(\ZZd,n,k) \cap \LP(\ZZd,n,k) &= \SP(\ZZd,n,k) \\
	\FG(\ZZd,n,k) \cap \OB(\ZZd,n,k) &= \langle \SP(\ZZd,n,k), \SHIFT(\ZZd,n,k) \rangle
	\end{align*}
	\end{restatable}

	As usual, the case $n = 1$ is not particularly interesting, and we have that $\FG(\ZZd,1,k) \cong \RTM(\ZZd,1,k)$. 
	In the general case the group is more complex. 
	
	We now prove that the $\FG(\ZZd,n,k)$-groups are non-amenable. In \cite{ElMo12}, a similar idea is used to prove that there exists a minimal $\ZZ^2$-subshift whose topological full group is not amenable.
	
	
	\begin{restatable}{proposition}{FreeProductEmbedding}
		Let $n \geq 2$. For all $m \in \NN$ we have that:
		$$\underset{m \text{ times}}{\underbrace{\ZZ/2\ZZ \ast \dots \ast \ZZ/2\ZZ}} \hookrightarrow \FG(\ZZd,n,k)$$
	\end{restatable}
	
	\begin{corollary}
		For $n \geq 2$, $\FG(\ZZd,n,k)$ and $\RTM(\ZZd,n,k)$ contain the free group on two elements. In particular, they are not amenable.
	\end{corollary}

	By standard marker constructions, one can also embed all finite groups and finitely generated abelian groups in $\FG(\ZZd,n,k)$ -- however, this group is residually finite, and thus does not contain $S_\infty$ or $(\QQ, +)$.
	
	\begin{proposition}
	Let $n \geq 2$ and $G$ be any finite group or a finitely generated abelian group. Then $G \leq \FG(\ZZd,n,k)$.
	\end{proposition}

	\begin{restatable}{theorem}{FGNotFinitelyGenerated}
	\label{thm:FGFinitelyGenerated}
	Let $n \geq 2, k \geq 1, d \geq 1$. Then the group $\FG(\ZZ^d,n,k)$ is residually finite and is not finitely generated.
	\end{restatable}
	
	The proof of this theorem is based on studying the action of the group on finite subshifts where heads are occur periodically. Non-finitely generatedness is obtained by looking at signs of permutations of the finitely many orbits, to obtain the \emph{sign homomorphism} to an infinitely generated abelian group. 

	\subsection{Elementary Turing machines and the LEF property of $\RTM$}
	
	\begin{definition}
		We define the group of elementary Turing machines $\EL(\ZZd,n,k) := \langle \FG(\ZZd,n,k),\LP(\ZZd,n,k)\rangle$. That is, the group generated by machines which only change the tape or move the head.
	\end{definition}
	
	
	
	\begin{restatable}{proposition}{Qp}
		Let $\QQ_p = \frac{1}{p}\ZZ$. Then $\alpha(\FG(\ZZd,n,k)) = \alpha(\EL(\ZZd,n,k)) = \QQ_k^d$. In particular, $\EL(\ZZd,n,k) \subsetneq \RTM(\ZZd,n,k)$. 
	\end{restatable}

	We do not know whether $\alpha(T) \in \ZZ^d$ implies $T \in \EL(\ZZd,n,1)$, nor whether $\EL(\ZZd,n,k)$ is finitely generated -- the sign homomorphism we use in the proof of finitely-generatedness of the group of finite-state automata does not extend to it.

	By the results of this section, the group $\RTM(\ZZd,n,k)$ is neither amenable nor residually finite. By adapting the proof of Theorem~\ref{thm:FGFinitelyGenerated}, one can show that it is locally embeddable in finite groups. See \cite{VeGo97,Zi02,We00} for the definitions.
	
	\begin{restatable}{theorem}{LEF}
	The group $\RTM(\ZZd,n,k)$ is LEF, and thus sofic, for all $n,k,d$.
	\end{restatable}

	\section{Computability aspects} 
	\label{section.computability}
	
	
	\subsection{Basic decidability results}

	First, we observe that basic management of local rules is decidable. Note that these results hold, and are easy to prove, even in higher dimensions.
	
	\begin{restatable}{lemma}{Decisions}
	Given two local rules $f, g : \Sigma^F \times Q \to \Sigma^F \times Q \times \ZZd$,
	\begin{itemize}
	\item it is decidable whether $T_f = T_g$,
	\item we can effectively compute a local rule for $T_f \circ T_g$,
	\item it is decidable whether $T_f$ is reversible, and
	\item we can effectively compute a local rule for $T_f^{-1}$ when $T_f$ is reversible.
	\end{itemize}
	\end{restatable}

	A group is called \emph{recursively presented} if one can algorithmically enumerate its elements, and all identities that hold between them. If one can furthermore decide whether a given identity holds in the group (equivalently, whether a given element is equal to the identity element), we say the group has a \emph{decidable word problem}. The above lemma is the algorithmic content of the following proposition:
	
	\begin{proposition}
	The groups $\TM(\ZZd,n,k)$ and $\RTM(\ZZd,n,k)$ are recursively presented and have decidable word problems in the standard presentations.
	\end{proposition}

	\subsection{The torsion problem}
	
	The \emph{torsion problem} of a recursively presented group $G$ is the set of presentations of elements $g \in G$ such that $g^n = 1_G$ for some $n \geq 1$. Torsion elements are recursively enumerable when the group $G$ is recursively presented, but the torsion problem need not be decidable even when $G$ has decidable word problem.
	

	In the case of $\RTM(\ZZd,n,k)$ the torsion problem is undecidable for $n \geq 2$. This result was shown by Kari and Ollinger in~\cite{KaOl08} using a reduction from the mortality problem which they also prove to be undecidable. 
	
	
	The question becomes quite interesting if we consider the subgroup $\FG(\ZZd,n,k)$ for $n \geq 2$, as then the decidability of the torsion problem is dimension-sensitive. 

\begin{restatable}{theorem}{OneDTorsion}
The torsion problem of $\FG(\ZZ,n,k)$ is decidable.
\end{restatable}

\begin{restatable}{theorem}{TwoDTorsion}
For all $n \geq 2, k \geq 1, d \geq 2$, there is a finitely generated subgroup of $\FG(\ZZd,n,k)$ whose torsion problem is undecidable.
\end{restatable}

	 \section{Acknowledgements}

	 The third author was supported by FONDECYT grant 3150552.

	\bibliographystyle{plain}
	
	\bibliography{bib}

\begin{thebibliography}{10}

\bibitem{GaOlTo15}
N.~Ollinger A.~Gajardo and R.~Torres-Avil{\'e}s.
\newblock The transitivity problem of {Turing} machines.
\newblock 2015.

\bibitem{AaGrSc15}
S.~{Aaronson}, D.~{Grier}, and L.~{Schaeffer}.
\newblock {The classification of reversible bit operations}.
\newblock {\em ArXiv e-prints}, April 2015.

\bibitem{AuBaSa14}
N.~{Aubrun}, S.~{Barbieri}, and M.~{Sablik}.
\newblock {A notion of effectiveness for subshifts on finitely generated
  groups}.
\newblock {\em ArXiv e-prints}, December 2014.

\bibitem{AuSa13}
N.~Aubrun and M.~Sablik.
\newblock Simulation of effective subshifts by two-dimensional subshifts of
  finite type.
\newblock {\em Acta applicandae mathematicae}, 126(1):35--63, 2013.

\bibitem{BeBl14}
J.~{Belk} and C.~{Bleak}.
\newblock {Some undecidability results for asynchronous transducers and the
  Brin-Thompson group 2V}.
\newblock {\em ArXiv e-prints}, May 2014.

\bibitem{BeMa14}
James Belk and Francesco Matucci.
\newblock Conjugacy and dynamics in {Thompson's} groups.
\newblock {\em Geometriae Dedicata}, 169(1):239--261, 2014.

\bibitem{Boykett16}
T.~Boykett, J.~Kari, and V.~Salo.
\newblock Strongly universal reversible gate sets.
\newblock {\em submitted}, 2016.

\bibitem{CeCo10}
T.~Ceccherini-Silberstein and M.~Coornaert.
\newblock {\em Cellular Automata and Groups}.
\newblock Springer Monographs in Mathematics. Springer-Verlag, 2010.

\bibitem{CzKa07}
E.~Czeizler and J.~Kari.
\newblock A tight linear bound on the synchronization delay of bijective
  automata.
\newblock {\em Theoretical Computer Science}, 380(1–2):23 -- 36, 2007.
\newblock Automata, Languages and Programming.

\bibitem{DuRoSh10}
B.~Durand, A.~Romashchenko, and A.~Shen.
\newblock Effective closed subshifts in 1d can be implemented in 2d.
\newblock In {\em Fields of logic and computation}, pages 208--226. Springer,
  2010.

\bibitem{ElMo12}
G.~{Elek} and N.~{Monod}.
\newblock {On the topological full group of a minimal {Cantor} Z\^{}2-system}.
\newblock {\em ArXiv e-prints}, December 2012.

\bibitem{GaGu10}
A.~Gajardo and P.~Guillon.
\newblock Zigzags in {Turing} machines.
\newblock In {\em Computer Science--Theory and Applications}, pages 109--119.
  Springer, 2010.

\bibitem{GaMa07}
A.~Gajardo and J.~Mazoyer.
\newblock One head machines from a symbolic approach.
\newblock {\em Theoretical Computer Science}, 370(1–3):34 -- 47, 2007.

\bibitem{GiPuSk99}
T.~Giordano, I.~Putnam, and C.~Skau.
\newblock Full groups of {Cantor} minimal systems.
\newblock {\em Israel Journal of Mathematics}, 111(1):285--320, 1999.

\bibitem{GrMe11}
R.~{Grigorchuk} and K.~{Medynets}.
\newblock {On algebraic properties of topological full groups}.
\newblock {\em ArXiv e-prints}, May 2011.

\bibitem{DeKuBl06}
P.~K{\r{u}}rka J~Delvenne and V.~Blondel.
\newblock Decidability and universality in symbolic dynamical systems.
\newblock {\em Fund. Inform.}, 74(4):463--490, 2006.

\bibitem{Je13}
E.~Jeandel.
\newblock Computability of the entropy of one-tape {Turing} machines.
\newblock {\em arXiv preprint arXiv:1302.1170}, 2013.

\bibitem{JuMo12}
K.~Juschenko and N.~Monod.
\newblock Cantor systems, piecewise translations and simple amenable groups.
\newblock {\em Annals of Mathematics}, 2012.

\bibitem{Ka96}
J.~Kari.
\newblock Representation of reversible cellular automata with block
  permutations.
\newblock {\em Theory of Computing Systems}, 29:47--61, 1996.
\newblock 10.1007/BF01201813.

\bibitem{Ka03}
J.~Kari.
\newblock {\em Developments in Language Theory: 6th International Conference,
  DLT 2002 Kyoto, Japan, September 18--21, 2002 Revised Papers}, chapter
  Infinite Snake Tiling Problems, pages 67--77.
\newblock Springer Berlin Heidelberg, Berlin, Heidelberg, 2003.

\bibitem{KaOl08}
J.~Kari and N.~Ollinger.
\newblock Periodicity and immortality in reversible computing.
\newblock In {\em Proceedings of the 33rd international symposium on
  Mathematical Foundations of Computer Science}, MFCS '08, pages 419--430,
  Berlin, Heidelberg, 2008. Springer-Verlag.

\bibitem{Ku97}
P.~K\r{u}rka.
\newblock On topological dynamics of {Turing} machines.
\newblock {\em Theor. Comput. Sci.}, 174(1-2):203--216, March 1997.

\bibitem{Ku10}
P.~K\r{u}rka.
\newblock Erratum to: Entropy of {Turing} machines with moving head.
\newblock {\em Theor. Comput. Sci.}, 411(31-33):2999--3000, June 2010.

\bibitem{LiMa95}
D.~Lind and B.~Marcus.
\newblock {\em An Introduction to Symbolic Dynamics and Coding}.
\newblock Cambridge University Press, 1995.

\bibitem{BoLiRu88}
D.~Lind M.~Boyle and D.~Rudolph.
\newblock The automorphism group of a shift of finite type.
\newblock {\em Transactions of the American Mathematical Society}, 306(1):pp.
  71--114, 1988.

\bibitem{Ma12}
H.~{Matui}.
\newblock {Topological full groups of one-sided shifts of finite type}.
\newblock {\em ArXiv e-prints}, October 2012.

\bibitem{SaScworkinprogress}
V.~Salo and M.~Schraudner.
\newblock in preparation.

\bibitem{SaTo15c}
V.~Salo and I.~T\"orm\"a.
\newblock Group-walking automata.
\newblock In Jarkko Kari, editor, {\em Cellular Automata and Discrete Complex
  Systems}, volume 9099 of {\em Lecture Notes in Computer Science}, pages
  224--237. Springer Berlin Heidelberg, 2015.

\bibitem{SaTo15}
V.~Salo and I.~T\"orm\"a.
\newblock Plane-walking automata.
\newblock In Teijiro Isokawa, Katsunobu Imai, Nobuyuki Matsui, Ferdinand Peper,
  and Hiroshi Umeo, editors, {\em Cellular Automata and Discrete Complex
  Systems}, volume 8996 of {\em Lecture Notes in Computer Science}, pages
  135--148. Springer International Publishing, 2015.

\bibitem{BlCaNi02}
J.~Cassaigne V.~Blondel and C.~Nichitiu.
\newblock On the presence of periodic configurations in {Turing} machines and
  in counter machines.
\newblock {\em Theoretical Computer Science}, 289(1):573--590, 2002.

\bibitem{VeGo97}
A.~Vershik and E~Gordon.
\newblock Groups that are locally embeddable in the class of finite groups.
\newblock {\em Algebra i Analiz}, 9(1):71--97, 1997.

\bibitem{We00}
B.~Weiss.
\newblock Sofic groups and dynamical systems.
\newblock {\em Sankhy{\=a}: The Indian Journal of Statistics, Series A}, pages
  350--359, 2000.

\bibitem{Xu15}
S.~Xu.
\newblock Reversible logic synthesis with minimal usage of ancilla bits.
\newblock Master's thesis, MIT, June 2015.

\bibitem{Zi02}
M~Ziman et~al.
\newblock On finite approximations of groups and algebras.
\newblock {\em Illinois Journal of Mathematics}, 46(3):837--839, 2002.

\end{thebibliography}

	\section*{Appendix: Proofs}

	\propCACharacterization*
	
	\begin{proof}
		Let $T \in \TM(\ZZd,n,k)$. $T$ is clearly shift invariant and continuous, therefore $T \in \End(X_{n,k})$. Also, $T$ acts trivially on $X_{n,0}$ so $T|_Y = \ID$ and if the initial configuration has a head, it can only move by a finite amount and not disappear, thus $T^{-1}(Y) \subset Y$. Moreover, if $T \in \RTM(\ZZd,n,k)$, then $T$ has a Turing machine inverse, thus a cellular automaton inverse, and it follows that $T \in \Aut(X_{n,k})$.
		
		Conversely, let $\phi \in \End(X_{n,k})$, so that $\phi(x,y)_{\vec v} = \Phi( \sigma_{-\vec v}((x,y))|_F)$ for some local rule $\Phi: (\Sigma \times \{0,\dots,k\})^{F} \to \Sigma \times \{0,\dots,k\}$ and $F$ a finite subset of $\ZZd$. As $\phi|_Y = \ID$, when $n \geq 2$ we can deduce that $\vec 0 \in F$ and that $\Phi(u,v) = (u_{\vec 0},v_{\vec 0})$ if $v = 0^F$. Therefore if $(x,y) \in X_{n,k}$, $y_{\vec v} \neq 0$ and we define $W_{\vec v} = \{\vec u \in \ZZd \mid \vec v \in \vec u+F \}$ we get that $\phi(x,y)|_{\ZZd \setminus W_{\vec v}} = (x,y)|_{\ZZd \setminus W_{\vec v}}$. We can thus extend $\Phi$ to $\widetilde{\Phi}: (\Sigma \times \{0,\dots,k\})^{W_{\vec 0}+F} \to  (\Sigma \times \{0,\dots,k\})^{W_{\vec 0}}$ defined by pointwise application of $\Phi$. We can then define $f_{\phi} : \Sigma^{W_0+F} \times Q \to \Sigma^{W_0} \times Q \times \ZZ$ by using $\widetilde{\Phi}$ as follows:  We set $f_{\phi}(p,q) = (p',q',\vec u)$ if, after defining $r \in \{0,\dots,k\}^{W_0+F}$ such that $r_{\vec 0} = q$ and $0$ elsewhere, we have $\widetilde{\Phi}(p,r) = (p',r')$ and $r' \in \{0,\dots,k\}^{W_0}$ contains the symbol $q'$ in position $\vec u$ (there is always a unique such position $\vec u$ as $\phi^{-1}(Y) \subset Y$).
		
		 It can be verified that the Turing machine $T_{f_{\phi}}$ is precisely $\phi$, therefore $\phi \in \TM(\ZZd,n,k)$. If $\phi \in \Aut(X_{n,k})$ then the $\phi^{-1}(Y) \subset Y$ property is implied by $\phi|_Y = \ID$, so such $\phi$ is also Turing machine. It is easy to see that if $\phi$ satisfies $\phi|_Y = \ID$, then also $\phi^{-1}|_Y = \ID$. It follows that the inverse map $\phi^{-1}$ is also a Turing machine. Thus, $\phi \in \RTM(\ZZd,n,k)$. \qed
	 \end{proof}

     \MovingTapeDynamicalDef*
	
	\begin{proof}
	It is easy to see that $T_f$ for any local rule $f : \Sigma^{F} \times Q \to \Sigma^{F'} \times Q \times \ZZd$ is continuous. The projection to the third component of $f$ gives the function $s$, and one can take the maximal length of a vector in $F'$ as $a$.
	
	For the converse, since $s$ is a continuous function from a compact space to a discrete one, $s(x, q)$ only depends on a finite set $F_0$ of coordinates of $x$ and obtains a maximum $m$. Since $T$ is continuous, $T(x,q)_{[-a-m,a+m]}$ depends only on a finitely set of coordinates $F_1$ of $x$. It is then easy to extract a local rule
	\[ f : \Sigma^{F_0 \cup F_1} \times Q \to \Sigma^{[-a,a]^d} \times Q \times \ZZd, \]
	for $T$. \qed
	\end{proof}
	
     \GroupIsomorphismLemma*
	
	\begin{proof}
		Consider again the epimorphism $\Psi$ and let $X \subset \{0,1\}^{\ZZd}$ be any strongly aperiodic subshift (that is, for each configuration $x \in X$, $stab_{\sigma}(x)= \{0\}$). Then if $x \in X$, $x \not\sim \sigma_{\vec v}(x)$ for any $\vec v \in \ZZd \setminus \{\vec 0\}$, otherwise the compactness of $X$ would allow us to construct a periodic point by shifting the finite set of differences to infinity. Let $\widetilde{x} \in X$ and consider $T \neq T'$ in $\TM(\ZZd,n,k)$ and a pair $(x,y)$ such that $T(x,y) \neq T'(x,y)$. As elements of $\TM(\ZZd,n,k)$ act locally around the head, we can modify $x$ outside a finite region such that it is asymptotically equivalent to $\widetilde{x}$ and call this modified version $x'$. We choose the finite region large enough to ensure $T(x',y) \neq T'(x',y)$. Obviously $x' \not\sim \sigma_{\vec v}(x')$ for non-zero $\vec v$ and thus if the non-zero symbol carried by $y$ is $q$ then $\Psi(T)(x',q) \neq \Psi(T')(x',q)$. \qed
	\end{proof}

     \InjectiveSurjectiveReversible*
	
	\begin{proof}
	Let $T$ be arbitrary, and let $F = F' = [-r,r]^d$ and $\epsilon = \frac{1}{k(2r+1)^d}$, where $r$ is the radius of $T$. Consider the cylinders $C_i = [p_i] \times \{q\}$ where $p_i \in \Sigma^F, q \in Q$. These cylinders form a clopen partition of $\Sigma^{\ZZd} \times Q$ into $k(2r+1)^d$ cylinders of measure $\epsilon$.
	
	Now, because $r$ is the radius of $T$, $T$ is a homeomorphism from $C_i$ onto $D_i = T(C_i)$, and $D_i$ is a cylinder set of the form $[p'] \times \{q'\}$ for some $p' \in \Sigma^{\vec v + F}$, $q \in Q$, which must be of the same measure as $C_i$ as the domain $\vec v + F$ of $p'$ has as many coordinates as the domain $F$ of $p$. Note that $D_i$ is not necessarily a cylinder centered at the origin, and the offset $\vec v$ is given by the shift-indicator. Now, observe that injectivity is equivalent to the cylinders $D_i = T(C_i)$ being disjoint. Namely, they must be disjoint if $T$ is injective, and if they are disjoint then $T$ is injective because $T|_{C_i} : C_i \to D_i$ is a homeomorphism. Surjectivity on the other hand is equal to $\Sigma^{\ZZd} \times Q = \bigcup_i D_i$, since $\bigcup_i D_i = \bigcup_i T(C_i) = T(\Sigma^{\ZZd} \times Q)$.
	
	Now, it is easy to show that injectivity and surjectivity are equivalent: If $T$ is injective, then the $D_i$ are disjoint, and $\mu(\bigcup_i D_i) = \sum_i \mu(D_i) = 1$, so we must have $\bigcup_i D_i = \Sigma^{\ZZd} \times Q$ because $\Sigma^{\ZZd} \times Q$ is the only clopen set of full measure. If $T$ is not injective, then for some $i \neq j$ we have $D_i \cap D_j \neq \emptyset$. Then $D = D_i \cap D_j$ is a nonempty clopen set, and thus has positive measure. It follows that $\mu(\bigcup_i D_i) \leq \sum_i \mu(D_i) - \mu(D) < 1$, so $\bigcup_i D_i \subsetneq \Sigma^{\ZZd} \times Q$. Of course, since injectivity and surjectivity are equivalent, they are both equivalent to bijectivity, and thus reversibility (by compactness).
	
	The proof shows that reversibility is equivalent to preserving the uniform Bernoulli measure in the forward sense -- if $T$ is reversible, then $\mu(T_f(A)) = \mu(A)$ for all clopen sets $A$, and thus for all Borel sets, while if $T$ is not reversible, then there is a disjoint union of cylinders $C \cup D$ such that $\mu(T(C \cup D)) < \mu(C \cup D)$.
	
	For measure-preservation in the usual (backward) sense, observe that the reverse of a reversible Turing machine is reversible and thus measure-preserving in the forward sense, so a reversible Turing machine must itself be measure-preserving in the traditional sense. If $T$ is not reversible, then $\mu(T(C \cup D)) < \mu(C \cup D)$ for some disjoint cylinders $C$ and $D$ large enough that $T|_C$ and $T|_D$ are measure-preserving homeomorphisms. Then for $E = T(C) \cap T(D)$ we have $\mu(T^{-1}(E)) \geq \mu((T^{-1}(E) \cap C) \cup (T^{-1}(E) \cap D)) = 2\mu(E)$. \qed
	\end{proof}

     \ClassicalReversibility*
	
	\begin{proof}
	We only need to show that if $T$ is reversible then it is of the stated form. Let $f_T : \Sigma \times Q \to \Sigma \times Q \times \{-1,0,1\}$ be a local rule for $T$ in the moving tape model. We claim that if $f_T(a,q) = (b,r,d)$ and $f_T(a',q') = (b',r,d')$ then $d = d'$. Namely, otherwise one can easily find two configurations with the same image. There are multiple cases to consider, but we only show $d = 0$ and $d' = 1$. In this case
	\[ T((xb'a.y, q)) = (xb'b.y, r) = T((xa'.by), q'). \]
	
	We have shown that the direction of movement is entirely determined by the state we enter. Of course, for $T$ to be injective, also $f_T$ must be injective, so the map $g : \Sigma \times Q \to \Sigma \times Q$ defined by $g(a,q) = (b,r)$ if $f_T(a,q) = (b,r,d)$ is injective, thus a bijection. Now, $T_0$ is defined as the permutation $g$, and $T_1$ as the finite-state automaton with local rule $f_{T_1}(a,q) = (a,q,d)$ if $f_T(b,q') = (b',q,d)$ for some $(b, q) \in \Sigma \times Q$. \qed
	\end{proof}

	\RTMIsNotFinitelyGenerated*
	
	
	\begin{proof}
		Consider the $(1,n,k)$-Turing machine $T_{\texttt{SURF},m}$ given by the local function $f : \Sigma^{\{0,\dots,m\}} \times Q \to \Sigma^{\{0,\dots,m\}} \times Q \times \ZZ$ given by the following: For $a \in \Sigma$ then $f(0^ma,q) = (a0^m,q,1)$. Otherwise $f(u,q)= (u,q,0)$. This machine is reversible, and satisfies that $\alpha(T_{\texttt{SURF},m}) = 1/n^{m}$. This machine can easily be extended to a $(d,n,k)$-Turing machine with average movement $(1/n^m,0,\dots,0)$. Suppose $\RTM(\ZZd,n,k)$ is generated by a finite set $S$. Then $\alpha(\RTM(\ZZd,n,k)) = \langle \alpha(S) \rangle$ which is a subgroup of $\QQ^d$ of elements which have their denominator bounded by the lowest common multiple of the denominators of $\alpha(T)$ for $T \in S$. As $n \geq 2$ there is $m \in \NN$ such that $\alpha(T_{\texttt{SURF},m}) \notin \alpha(\RTM(\ZZd,n,k))$, thus $T_{\texttt{SURF},m} \notin \RTM(\ZZd,n,k)$, which yields a contradiction. \qed
	\end{proof}

	\LPNotResiduallyFinite*
	
	\begin{proof}
		It suffices to show the result for $d = 1$. Let $\sigma \in S_{\infty}$ with support $F \subset N$ and let $T_{\sigma}$ be given by the local function $f_{\sigma} : \Sigma^F \times Q \to \Sigma^F \times Q \times \ZZ$ defined by $f_{\sigma}(p,q) = (p',q,0)$ where $p'_m = p_{\sigma(m)}$. By definition it's clear that $T_{\sigma_1} \circ T_{\sigma_2} = T_{\sigma_1 \circ\sigma_2}$. Moreover the homomorphism is injective: If $\sigma_1 \neq \sigma_2$ then there exists $m \in \NN$ such that $\sigma_1(m) \neq \sigma_2(m)$. As $n \geq 2$, we can consider the configuration $(x,y) \in X_{n,k}$ such that $x_m = y_0 = 1$ and $x_{\ZZ \setminus \{m\}}$ and $y_{\ZZ \setminus \{0\}}$ contain only zeroes. 
		Then $T_{\sigma_1}(x,y) \neq T_{\sigma_2}(x,y)$. \qed
	\end{proof}

	\OBAmenable*
	
	\begin{proof}
	Clearly, the image of an element of $\OB(\ZZd,n,k)$ in the average movement homomorphism $\alpha$ simply records the powers of the generators $T_{e_i}$ in the representation of the element. The image $\ZZ^d$ of this homomorphism is abelian, and thus amenable. The kernel is $\mbox{Ker}(\alpha) \cap \OB(\ZZd,n,k) = \LP(\ZZd,n,k)$ which is locally finite, thus amenable. Thus, $\OB(\ZZd,n,k)$ is the extension of an amenable group by an amenable group, thus amenable. \qed
	\end{proof}

	\LamplighterEmbedding*
	
	\begin{proof}
		Let $|G| = m$. 
		Let $T_z \in \OB(\ZZd,2,1)$ be the machine which always moves the head by $z$ and for $s \in S_m$ let $T_{s}$ the immersion into $\ZZd$ of the machine defined in Proposition~\ref{proposition_LP_not_residually_finite}. By Cayley's theorem $G$ is isomorphic to a subgroup of $S_m$ generated by some elements $s_1,\dots,s_\ell$. Let $e_1,\dots,e_d$ be the generators of $\ZZd$. Then $G \wr \ZZ^d$ is isomorphic to $\langle T_{s_1},\dots, T_{s_\ell}, T_{me_1},\dots,T_{me_d}\rangle$. Indeed, $\langle T_{me_1},\dots,T_{me_d}\rangle \cong m\ZZd$ and the action of the first $\ell$ machines only acts on the coset $\ZZd/m\ZZd$ and is isomorphic to $G$. \qed
	\end{proof}

	The proof of Theorem~\ref{thm:OBFG} is Appendix~B.

	\Orthogonal*
	
	\begin{proof}
	The inclusions $\SP(\ZZd,n,k) \subset \FG(\ZZd,n,k) \cap \LP(\ZZd,n,k)$ and
	\[ \langle \SP(\ZZd,n,k), \SHIFT(\ZZd,n,k) \rangle \subset \FG(\ZZd,n,k) \cap \OB(\ZZd,n,k) \]
	follow directly from the definitions. For the converse inclusions, observe that an element of $\FG(\ZZd,n,k) \cap \LP(\ZZd,n,k)$ cannot modify the tape or move the head, so it can only perform a permutation of the state as a function of the tape. In $\FG(\ZZd,n,k) \cap \OB(\ZZd,n,k)$, precisely the unconditional shifts have been added.
	\end{proof}

	\FreeProductEmbedding*
	
	\begin{proof}
	We show this result only for $\FG(\ZZ,n,k)$ where $n = m$. 
	For $a \in \Sigma$ define $T_a \in \FG(\ZZ,m,k)$ with radius $1$ 
	by the following rules:
		
	\begin{enumerate}
		\item if $x_0 = a$ and $x_1 \neq a$ move the head to the right.
		\item if $x_{-1} = a$ and $x_0 \neq a$ move the head to the left.
		\item Otherwise stay in place.
	\end{enumerate}
		
	The machine $T_a$ does not change the tape, so it belongs to $\FG(\ZZ,m,k)$. It is clearly reversible as $T_a^2 = \ID$. We claim that $\langle T_1,\dots,T_m \rangle$ is isomorphic to the free product of $m$ copies of $\ZZ/2\ZZ$. Indeed, every element in $\ZZ/2\ZZ * \dots * \ZZ/2\ZZ$ can be represented as a word in $\{a_1,\dots,a_m\}$ where no factor $a_ia_i$ appears. Given a word $w = w_1\dots w_t$ in that form we can construct a configuration $x(w)$ such that $x(w)_{j-1} = i$ if $w_j = a_i$ for $j \in \{1,\dots, t\}$, $x(w)_{t} \neq x(w)_{t-1}$ and $x(w)_j = 1$ otherwise. The machine $T_{w_t}\circ \dots\circ T_{w_1}$ acts on $x(w)$ by moving the head $t$ steps to the right. Therefore it is not the identity. \qed
	\end{proof}

\FGNotFinitelyGenerated*

	\begin{proof}	
	Let $Y_m \subset (\Sigma \times (Q \cup \{0\}))^{\ZZd}$ be the finite set of points $y$ such that
	\begin{itemize}
	\item $\sigma_{m \vec v}(y) = y$ for unit vectors $\vec v \in \ZZd$, and
	\item for some $\vec u \in \ZZd$, $(y_{\vec v})_2 \neq 0 \iff \vec v \in \vec u + m \ZZd$.
	\end{itemize}
	In other words, configurations of $Y_{m}$ have period $m$ in every direction, and have Turing machine heads on a sublattice of the same periods.
	
	Now, to each $T \in \FG(\ZZ^d,n,k)$ we associate the map $\phi(T) : Y_m \to Y_m$ that applies the local rule of the Turing machine at each head position, ignoring the other heads. Since $T$ is injective and never modifies the tape, also the map $\phi(T)$ is injective, and thus $\phi : \FG(\ZZ^d,n,k) \to \Sym(Y_m)$ is a homomorphism to the permutation group $\Sym(Y_m)$ of $Y_m$.
	
	To show the group is not finitely generated, consider the signs of the permutations $\phi(T)$ performs on $Y_m$ for different $m$. It is easy to show that for any vector of signs $s_1, s_2, s_3, \ldots \in \{-1,1\}^\NN$ where $s_i = 1$ for all large enough $i$, we can construct a finite-state automaton $T$ such that for all $m$, $\phi(T)$ performs a permutation with sign $s_m$ on $Y_m$. This means that $\FG(\ZZ^d,n,k)$ has as a homomorphism image the non-finitely generated group $(\ZZ/2\ZZ)^{\infty}$, so $\FG(\ZZ^d,n,k)$ itself cannot be finitely generated. \qed
	\end{proof}

	\Qp*
	
	\begin{proof}
	The machine $T_j$ that increments the state by $1$ on each step (modulo the number of states) and walks one step along the $j$th axis whenever it enters the state $1$, has $\alpha(T_j) = (0,\ldots,0,1/k,0,\ldots,0)$. We obtain
	\[ \langle \alpha(T_j) \;|\; j \leq d \rangle = \QQ_k^d \]
	
	Next, let us show that for every finite-state machine $T$, we have $\alpha(T) \in \QQ_k^d$. For this, consider the behavior of $T$ on the all-zero configuration. Given a fixed state $q$, $T$ moves by an integer vector $\vec v_q$, thus contributing $\frac{1}{k}\vec v_q$ to the average movement. Let $\vec{v} = \sum_{q \in Q}\frac{1}{k}\vec v_q$ be the average movement of $T$ on the all-zero configuration. 
	
	We claim that $\alpha(T) = \vec v$. Note that by composing $T$ with a suitable combination of the machines $T_{j}$ and their inverses, it is enough to prove this in the case $\vec v = \vec{0}$. Now, for a large $m$, Let $p \in \Sigma^{[-m,m]^d}$ be a pattern, $\vec{u} \in [-m,m]^d$ a position and $q \in Q$ a state. Complete $p$ to a configuration $x_p \in \Sigma^{\ZZd}$ by writing $0$ in every cell outside $[-m,m]^d$. Write $\alpha_m(T)$ for the average movement of $T$ for the finitely many choices of $p, u, q$. Formally, if $s_T$ is the shift indicator of $T$: $$\alpha_m(T) = \frac{1}{k (2m+1)^d n^{(2m+1)^d} } \sum_{p,u,q} s_T(\sigma_{-\vec{u}}(x_p),q).$$ As $m \rightarrow \infty$, it is easy to show that $\alpha_m(T) \rightarrow \alpha(T)$, as the movement vector of $T$ is distributed correctly in all positions except at the border of the hypercube which grows slower than the interior.
	
	On the other hand, it is easy to show that for any fixed $[-m,m]^d$-pattern $p$, then the average movement of $T$ on $x_p$ started from a random state and a random position is precisely $\vec 0$, that is, $\sum_{u,q} s_T(\sigma_{-\vec{u}}(x_p),q) = \vec{0}$. This follows from the fact that $T \in \FG(\ZZ^d,n,k)$ and thus the action is simply a permutation of the set of position-state pairs and the fact that $\vec{v}= \vec{0}$. From here we conclude that the sum restricted to $u \in [-m,m]^d$ is $o(m^d)$. It follows that $\alpha(T) = \lim \alpha_m(T) = \vec 0$. \qed
	\end{proof}

	\LEF*															

	\begin{proof}
	The proof we give is essentially the same proof as that of residual finiteness of $\FG(\ZZd,n,k)$, with the difference that we only obtain a `local' homomorphism with that construction.

	Let $M \subset \RTM(\ZZd,n,k)$ be a finite set of Turing machines all of which have radius at most $r$. Let $Y_{8r}$ be as in the proof of the previous theorem.
	
	Now, to each $T \in M^2$ we associate the map $\phi(T) : Y_{8r} \to Y_{8r}$ that applies the local rule of the Turing machine at each head position. The radius of every Turing machine in $M^2$ is at most $2r$, so it is easy to see that since $T$ is injective, also $\phi(T)$ is, and thus $\phi : M^2 \to \Sym(Y_{8r})$ is a function from $M^2$ to the permutation group of $Y_{8r}$. As the heads do not interact in the first two steps, we have $\phi(T \circ T') = \phi(T) \circ \phi(T')$ for all $T,T' \in M$. \qed
	\end{proof}

	\Decisions*
	
	\begin{proof}
For the first claim, simply minimize each rule: iteratively replace $F$ by $F' = F \setminus \{a\}$ as long as $f$ does not depend on the coordinate $a$ in the sense that $f(p,q)_a = p_a$ for all $(p, q) \in \Sigma^F \times Q$. This minimization process ends after at most $|F|$ steps and if $T_f = T_g$ the same local rule is reached from both $f$ and $g$, while this of course does not happen if $T_f \neq T_g$.

Finding a local rule for the composition of two Turing machines is a simple exercise.

For the decidability of reversibility, we give a semialgorithm for both directions. First, if $T_f$ is reversible, then it has a reverse $T_g$. We thus only need to enumerate local rules $g$, and check whether $T_f \circ T_g = \ID$, which is decidable by the previous lemmas.
	
	If $T_f$ is not reversible, then $T_f(x,y) = T_f(x',y')$ for some $(x,y), (x',y') \in \Sigma^{\ZZd} \times X_k$ with $y_{\vec 0} \neq 0$. If $r$ is the move-radius of $f$, then necessarily the nonzero position of $y'$ is at distance at most $r$ from the origin each other and $x_{\vec v} = x'_{\vec v}$ for $|\vec v|$ larger than the radius of $T_f$. Then we can assume $x_{\vec v} = x'_{\vec v} = 0$ for all such $\vec v$. It follows that if $T_f$ is not injective, it is not injective on the finite set of configurations $(x,y) \in \Sigma^{\ZZd} \times X_k$ where $(x,y)_{\vec v} = 0$ for all $|\vec v|$ larger than the radius of $T_f$, which we can check algorithmically.
	
	Of course, if $T$ is reversible, we find a reverse for it by enumerating all Turing machines and outputting the first $T'$ such that $T \circ T' = T' \circ T = \ID$, which we can check by the decidability of equality of Turing machines. \qed
	\end{proof}																						
	
	\OneDTorsion*

\begin{proof}
A semialgorithm for recognizing torsion elements exists in all dimensions, since the word problem is decidable: Given $T \in \FG(\ZZ,n,k)$, for $n = 1, 2, 3, \ldots$, check whether $T^n = \ID$. If this happens for some $n$, $T$ is a torsion element.

For the other direction we give a proof with a dynamical flavor. We need to show that there is a spatially periodic point where, from some initial state, the head of $T$ performs an infinite periodic walk either left or right. A semialgorithm can then enumerate periodic points and initial states, and once it finds a point where the head moves to infinity, it has proved that $T$ is not a torsion element. First, observe that by compactness, there must be a point and an initial state where the head walks arbitrarily far from its initial position. We may suppose that $T$ can walk arbitrarily far to the right.


Now, consider a walk of the head of $T$ from position $0$ to some position $j >> 0$, given by some configuration $(x,y)$. Because $T$ has finite move-radius, there must be a syndetic set of positions $h \in [0, j]$ that $T$ visits on the way to $j$, such that there is a maximal $\ell_h$ such that $(T^{\ell_h}(x,y)_2)_{h} \neq 0$, and $(T^{\ell}(x,y)_2)_i = 0$ for all $\ell > \ell_h$, $i \leq h$. By the pigeonhole principle,
\[ (T^{\ell_h}(x,y)_2)_{h+[-r,r]} = (T^{\ell_{h'}}(x,y)_2)_{h'+[-r,r]} \]
for some $h < h'$, and then $(x_{[h-r, h'-r-1]})^\ZZ$ is a spatially periodic point where, started from the inital state $(T^{\ell_h}(x,y)_2)_{h}$, the head of $T$ performs a periodic walk to infinity.
\qed
\end{proof}																									

\TwoDTorsion*

\begin{proof}
We assume $d = 2$ -- in the general case, we simply walk on a $2$-dimensional plane of the configuration.

First, let us explain why, for a suitable alphabet $\Sigma$ and the local rule of an element $f \in \FG(\ZZ^2,|\Sigma|,2)$, it is undecidable whether $f$ is a torsion element in that group. We then prove that the alphabet can be fixed, then that we can restrict to a finitely generated subgroup as long as we have $4$ states, and finally we get rid of states altogether.

In \cite{Ka03}, it is shown that, given a set of Wang tiles $T$ and a function $d : T \to D$ where $D = \{(1,0),(-1,0),(0,1),(0,-1)\}$, the \emph{snake tiling problem} is undecidable, that is, it is undecidable whether it is possible to find a partial tiling $\tau : \ZZ^2 \to \{\epsilon\} \cup T$, and a path $p : \ZZ \to \ZZ^2$ such that $\tau(p(n)) \in T$ and $p(n+1) - p(n) = d(\tau(p(n)))$ for all $n$ and all tiles match their non-$\epsilon$ neighbors.

Let $\Sigma = T$. We use the state -- the \emph{direction bit} -- to denote a direction for our element $f \in \FG(\ZZ^2,|T|,2)$. As long as $f$ is on a tile that matches all its four neighbors, it walks along the vectors given by $d$ (or their negations, depending on the direction bit). If the neighbors do not match, the direction bit of $f$ is flipped.

Now, if it is possible to tile an infinite snake, then there is a configuration where $f$ walks to infinity. If it is impossible, then there is a bound on how far $f$ can walk from its starting position before turning back, and $f$ has finite order. This concludes the proof of undecidability when the alphabet can be chosen freely.

To get a fixed alphabet, encode the tiles into binary squares of size $n \times n$. By having the top left corner contain $\begin{smallmatrix}1&1\\1&1\end{smallmatrix}$ and having no two adjacent $1$s elsewhere, we ensure that there is a unique way to `parse' a given tiling into encodings of squares. Of course, our machine will simply flip its direction bit when it does not locally see a valid coding of a tile.

Next, we want to do the same construction with a finitely generated subgroup. For this, we add a second \emph{auxiliary bit}, or \emph{aux bit} to the state. We also change the alphabet to one allowing us to draw paths on configurations: every tile has zero, one or two incoming arrows from the cardinal directions $D$, and an outgoing path for each incoming one. This alphabet $B$ can be used to encode tiles into $(n \times n)$-squares as well. We assume that the encoded tiles are connected by a path to their neighbors according to the $d$-function, so that by following the path for $n$ steps from the central cell of the encoding of a Wang tile, we reach the central cell of the next Wang tile.


In our finite set of generators we take
\begin{enumerate}
\item $T_{\mathtt{walk}}$ that walks along the path depending on the direction bit,
\item $T_{\vec v}$ that walk in the direction $\vec v \in D$ independently of the configuration,
\item $g_s$ that flips the direction bit if the current cell is $s \in \Sigma$,
\item $h_s$ that flips the aux bit if the current cell is $s$,
\item $g_{+,s}$ that adds the aux bit to the direction bit if the current cell is $s$, and
\item $h_{+,s}$ that adds the direction bit to the aux bit if the current cell is $s$,
\end{enumerate}
Let $F$ be this set of machines.

More generally, for a pattern $p$ write $g_p$ and $h_p$ for the machines that flip the direction bit or the auxiliary bit if they see the pattern $p$. Observe first that it is enough to construct all such $g_p$ and $h_p$ out of the machines in $F$. Namely, Let $\mathcal{P}$ be the set of all the $3n \times 3n$ patterns $p$ that are not a legal coding of matching Wang tiles, and apply $g_p$ for all of them to get a machine $g = \prod_{p \in \mathcal{P}} g_p$. Then $T_{\mathtt{walk}}^n \circ g$ is a torsion element if and only if the Wang tile set allows no infinite snakes.


To build $g_p$ and $h_p$, we proceed by induction. Suppose $D(p) = D(p') \cup \{\vec v\}$, $p|_{D(p')} = p'$, $p_{\vec v} = s$ and $g_{p'}$ and $h_{p'}$ can be built from $F$. Then
\[ g_p = (T_{-\vec v} \circ g_{+,s} \circ T_{\vec v} \circ h_{p'})^2, \mbox{ and } h_p = (T_{-\vec v} \circ h_{+,s} \circ T_{\vec v} \circ g_{p'})^2. \]

Finally, let us get rid of the state. For this, add the symbol $0$ to the alphabet $B$. The head now only moves if it is on a square of the form $\begin{smallmatrix} b & 0 \\ 0 & 0 \end{smallmatrix}$ for $b \in B$ which is surrounded by squares of the same form. It uses the four positions modulo $\ZZ/2\ZZ \times \ZZ/2\ZZ$ in this square to denote its state. If the $2 \times 2$ period where zeroes and nonzeroes occur breaks, the direction bit is reversed (that is, the head moves to another position modulo $\ZZ/2\ZZ \times \ZZ/2\ZZ$). It is easy to modify the generators $F$ to obtain such behavior. Again, we obtain that that solving the torsion in the group generated by these finite-state machines requires solving the snake tiling problem for all sets of directed Wang tiles, and is thus undecidable. \qed
\end{proof}

\section*{Appendix B: $\OB$ is finitely generated}

Next we prove that $\OB(\ZZd,n,k)$ is finitely generated.
This is based on the existence of strongly universal reversible gates for permutations of $A^m$, recently proved for the binary
alphabet $A=\{0,1\}$ in~\cite{AaGrSc15,Xu15}, and generalized to other alphabets in~\cite{Boykett16}. We need a finite generating set
for permutations of $Q\times \Sigma^m$, and hence the proof in~\cite{Boykett16} has to be adjusted
to account for non-homogeneous alphabet sizes (that is, due to possibly having $n\neq k$).

The case $n=1$ is trivial: Group $\LP(\ZZd,1,k)$ is finite and
$\SHIFT(\ZZd,1,k)$ is generated by the single step moves.
We hence assume that $n>1$.

The following lemma was proved in~\cite{Boykett16}
(Lemmas 3 and 5):

\begin{lemma}
\label{lem:graphlemma}
Let $H=(V,E)$ be a connected undirected graph.
\begin{enumerate}
\item[(a)] The transpositions $(s\; t)$ for $\{s,t\}\in E$ generate $\Sym(V)$, the set of permutations of the vertex
set.
\item[(b)] Let $\Delta\subseteq \Sym(V)$ be a set of permutations of $V$ that
contains for each edge  $\{s,t\}\in E$ a 3-cycle $(x\; y\; z)$ where $\{s,t\}\subset \{x,y,z\}$. Then $\Delta$ generates
$\Alt(V)$, the set of even permutations of the vertex set.
\end{enumerate}
\end{lemma}

Let $m\geq 1$, and consider permutations
of $Q\times\Sigma^m$. \emph{Controlled swaps} are transpositions $(s\; t)$ where $s,t\in Q\times \Sigma^m$
have Hamming distance one. \emph{Controlled 3-cycles\/} are permutations $(s\; t\; u)$ where the Hamming distances between the three
pairs are 1, 1 and 2, that is, the vectors $s, t$ and $u$
are of the forms $uavcw$, $uavdw$ and $ubvdw$, where $a,b,c$ and $d$ are single letters.
Let us denote by $C_m^{(2)}$ and  $C_m^{(3)}$ the sets of controlled swaps and 3-cycles, respectively, in
$\Sym(Q\times\Sigma^m)$.
Let $H=(V,E)$ be the graph with vertices $V=Q\times\Sigma^m$ and
edges $\{s,t\}$ that connect elements $s$ and $t$ having Hamming distance one. This graph is clearly connected, so we get from
Lemma~\ref{lem:graphlemma}(b) that
the controlled 3-cycles generate all its even permutations:

\begin{lemma}
\label{lem:controlled}
Let $n>1$ and $m\geq 1$. The group $\Alt(Q\times \Sigma^m)$  is generated by $C_m^{(3)}$.
\end{lemma}

Let $\ell \leq m$, and let $f$ be a permutation of  $Q\times \Sigma^{\ell}$. We can apply $f$ on
$1+\ell$ coordinates of $Q\times \Sigma^{m}$,  while leaving the other
$m-\ell$ coordinates untouched. More precisely, the \emph{prefix application} $\hat{f}$ of $f$ on $Q\times \Sigma^{m}$, defined by
$$
\hat{f} : (q,s_1,\dots ,s_\ell,\dots ,s_m) \mapsto
(f(q,s_1,\dots ,s_\ell)_1, \dots ,f(q,s_1,\dots ,s_\ell)_\ell, s_{\ell+1},\dots ,s_m),
$$
applies $f$ on the first $1+\ell$ coordinates. To apply it on other choices of coordinates we conjugate $\hat{f}$
using rewirings of symbols.
For any permutation $\pi\in \Sym(\{1,\dots,m\})$ we define the
\emph{rewiring} permutation of $Q\times \Sigma^m$ by
$$
r_\pi : (q,s_1,\dots ,s_m) \mapsto (q,s_{\pi(1)},\dots ,s_{\pi(m)}).
$$
It permutes the positions of the $m$ tape symbols according to $\pi$.
Now we can conjugate the prefix application $\hat{f}$ using a rewiring to get
$\hat{f}_\pi = r_\pi^{-1} \circ \hat{f}\circ  r_\pi$, an
\emph{application} of $f$ in selected coordinates. Let us denote by
$$[f]_m = \{ \hat{f}_\pi\ |\ \pi\in \Sym(m)\}$$
the set of permutations of $Q\times \Sigma^{m}$ that are applications of $f$.
For a set $P$ of permutations we denote by $[P]_m$ the union of $[f]_m$ over all $f\in P$.

Note that if $n$ is even and $f\in\Sym(Q\times \Sigma^\ell)$ for $\ell < m$ then $[f]_m$ only contains even permutations.
The reason is that
the coordinates not participating in the application of $f$ carry a symbol of the even
alphabet $\Sigma$. The application $[f]_m$ then consists of an even number of
disjoint permutations of equal parity -- hence the result is even. In contrast,
for the analogous reason, if $n$ is odd then $[f]_m$ only contains odd permutations
whenever $f$ is itself is an odd permutation.

\begin{lemma}
\label{lem:iteration}
Let $m\geq 6$, and let $G_m =\langle [C_4^{(2)}]_m\rangle$ be the group generated by the
applications of controlled swaps of $Q\times \Sigma^4$ on $Q\times \Sigma^m$.
If $n=|\Sigma|$ is odd then $G_m=\Sym(Q\times \Sigma^{m})$. If $n$ is even
then $G_m=\Alt(Q\times \Sigma^{m})$.
\end{lemma}

\begin{proof}
For even $n$, by the note above,  $[C_4^{(2)}]_m \subseteq \Alt(Q\times
\Sigma^{m})$, and for odd $n$ there are odd permutations in $[C_4^{(2)}]_m$.
So in both cases
it is enough to show $\Alt(Q\times \Sigma^{m}) \subseteq G_m$. We
also note that, obviously, $[G_{m-1}]_m\subseteq G_m$.

Based on the
decomposition in Figure~\ref{fig:gates}, we first conclude that
any controlled 3-cycle $f$ of $Q\times \Sigma^m$ is a composition
of four applications of controlled swaps of $Q\times \Sigma^{m-2}$.
In the figure, the components of $Q\times \Sigma^m$ have been
ordered in parallel horizontal wires, the $Q$-component being among the topmost three
wires. Referring to the symbols in the illustration, the gate on the left is a
generic 3-cycle $(pszabcdw\; ptzabcdw\;
qszabcdw)$ where one of the first three wires is the $Q$-component,
$a,b,c,d\in \Sigma$ and $w\in \Sigma^{m-6}$.
The proposed decomposition consists of two different controlled swaps $p_1$ and $p_2$ applied twice in
the order $f=p_1p_2p_1p_2$. Because $p_1$ and $p_2$ are involutions,
the decomposition amounts to identity unless the input is of the form $xyz abcd w$ where
$x\in\{p,q\}$ and $y\in\{s,t\}$. When the input is of this form, it
is easy to very that the circuit on the right indeed amounts to the required 3-cycle. We
conclude that $C_m^{(3)}\subseteq  \langle [C_{m-2}^{(2)}]_m \rangle$, for all
$m\geq 6$. By Lemma~\ref{lem:controlled},
\begin{equation}
\label{eq:gateq}
\Alt(Q\times \Sigma^m) = \langle C_m^{(3)} \rangle \subseteq \langle
[C_{m-2}^{(2)}]_m \rangle.
\end{equation}

\begin{figure}[ht]
\begin{center}
\includegraphics[ width=11.5cm]{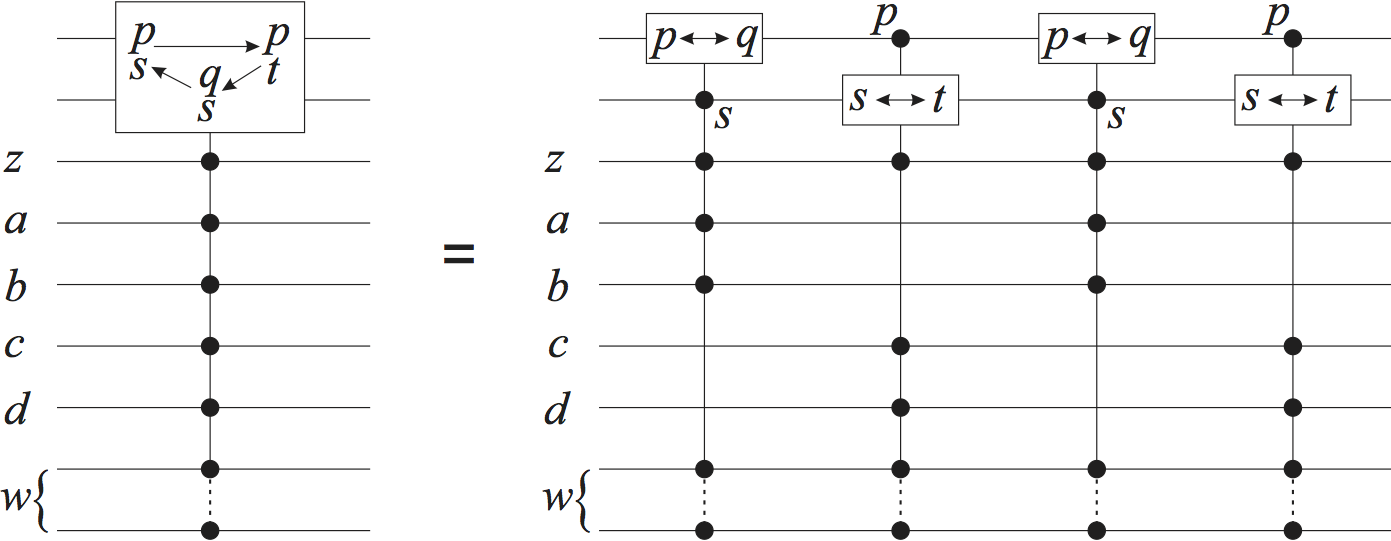}
\end{center}
\caption{A decomposition of a controlled 3-cycle of $Q\times \Sigma^m$ on the left into a sequence of
four applications of
controlled swaps of $Q\times \Sigma^{m-2}$ on the right. The ordering of the wires is such that
topmost three wires contain the $Q$-component and the two wires
changed by the 3-cycle (one of which may or may not be the $Q$-component). Black circles are control points: the gate computes the identity unless the
wire carries the symbol indicated at the left of the wire or next
to the control point.
}
\label{fig:gates}
\end{figure}

We proceed by induction on $m$. The base case $m=6$ is clear: By
(\ref{eq:gateq}),
$$\Alt(Q\times \Sigma^6) \subseteq \langle [C_{4}^{(2)}]_6 \rangle =
G_6.$$

Consider then $m>6$ and suppose
that $G_{m-1}$ is as claimed. If $n$ is odd then, by the inductive hypothesis,
$$
[C_{m-2}^{(2)}]_m  \subseteq  [\Sym(Q\times \Sigma^{m-1})]_m \subseteq
 [G_{m-1}]_m \subseteq G_m.
$$
By (\ref{eq:gateq}) then $\Alt(Q\times \Sigma^m)\subseteq
\langle [C_{m-2}^{(2)}]_m \rangle \subseteq  G_m$.
As pointed out above, $G_m$ contains odd
permutations (all elements of $[C_4^{(2)}]_m$ are odd), so $G_m=\Sym(Q\times \Sigma^{m})$
as claimed.

If $n$ is even then
an application of a permutation of $Q\times \Sigma^{m-2}$ on
$Q\times \Sigma^m$ is also an application of an even permutation of $Q\times \Sigma^{m-1}$
on $Q\times \Sigma^m$. (For this reason we left two wires
non-controlling for the gates on the
right side of Figure~\ref{fig:gates}.) By this and the inductive hypotheses,
$$
[C_{m-2}^{(2)}]_m \subseteq [\Alt(Q\times \Sigma^{m-1})]_m\subseteq
[G_{m-1}]_m \subseteq G_m,
$$
so, by (\ref{eq:gateq}), we have the required $\Alt(Q\times \Sigma^m)\subseteq
G_{m}$.
\qed

\end{proof}

\begin{corollary}
\label{cor:borrowedbit}
$[\Sym(Q\times \Sigma^{m})]_{m+1} \subseteq \langle [\Sym(Q\times \Sigma^4)]_{m+1}\rangle$
for all $m\geq 5$.
\end{corollary}

\begin{proof}
If $n$ is even then $[\Sym(Q\times \Sigma^{m})]_{m+1} \subseteq \Alt(Q\times
\Sigma^{m+1})$ and if $n$ is odd then
$[\Sym(Q\times \Sigma^{m})]_{m+1} \subseteq \Sym(Q\times
\Sigma^{m+1})$. In either case, the claim follows from
Lemma~\ref{lem:iteration} and $C_4^{(2)}\subseteq \Sym(Q\times
\Sigma^4)$.
\qed
\end{proof}

In Corollary~\ref{cor:borrowedbit}, arbitrary permutations of $Q\times\Sigma^{m}$
are obtained as projections of permutations of
$Q\times\Sigma^{m+1}$.
The extra symbol is an ancilla that can have an arbitrary initial
value and is returned back to this value in the end.
Such an ancilla is called a ``borrowed bit" in~\cite{Xu15}.
It is needed in the case of even $n$
to facilitate implementing odd permutations of $Q\times \Sigma^{m}$.

Now we are ready to prove the following theorem.

    \OBFG*
    
\begin{proof}
We construct a finite generating set $A_1\cup A_2\cup A_3$.

Let $A_1$ contain the one step moves $T_{e_i}$ for $i=1,\dots ,d$. These clearly generate
$\SHIFT(\ZZd,n,k)$.

Each $T\in \LP(\ZZd,n,k)$
is defined by a local rule $f:\Sigma^F\times Q \rightarrow  \Sigma^F\times Q\times\{0\}$
with a finite $F\subseteq \ZZd$. To have injectivity, we clearly need that $\pi:(p,q)\mapsto (f(p,q)_1, f(p,q)_2)$
is a permutation of $\Sigma^F\times Q$. We denote $T=P_\pi$.
Let us fix an arbitrary $E\subseteq  \ZZd$ of size $4$, and let
$A_2$ be the set of all $P_\pi\in \LP(\ZZd,n,k)$ determined by $\pi\in \Sym(\Sigma^E\times Q)$.

For any  permutation $\alpha$ of $\ZZd$ with finite support, we define
the cell permutation machine $C_\alpha: (p,q) \mapsto (p',q)$, where $p'_{\vec v}=p_{\alpha({\vec v})}$
for all ${\vec v}\in\ZZd$. These are clearly in $\LP(\ZZd,n,k)$. We take $A_3$ to consist of the cell
permutation machines $C_i=C_{(0\; e_i)}$ that, for each $i=1,\dots ,d$, swaps the contents of the currently scanned cell and its
neighbor with offset $e_i$.

Observe that $A_1$ and $A_3$ generate all cell permutation
machines $C_\alpha$. First, conjugating $C_i$ with
$T_{\vec v}\in\SHIFT(\ZZd,n,k)$ gives the cell permutation machine $C_\alpha=T_{\vec v}^{-1}C_iT_{\vec
v}$ for the transposition $\alpha=({\vec v}\; {\vec v} + e_i)$.
Such transpositions generate all permutations of $\ZZd$ with finite
support: This follows from Lemma~\ref{lem:graphlemma}(a) by
considering a finite connected grid graph containing the support of
the permutation.

Consider then an arbitrary
$P_\pi\in \OB(\ZZd,n,k)$, where $\pi\in \Sym(\Sigma^F\times Q)$.
We can safely assume $|F|\geq 5$.
Let us pick one ancilla $v\in\ZZd\setminus F$ and denote $F'=F\cup \{v\}$.
By Corollary~\ref{cor:borrowedbit}, $P_\pi$ is a composition of
machines of type $P_\rho$ for $\rho\in \Sym(\Sigma^H\times Q)$
where $H\subseteq F'$ has size $|H|=4$. It is enough to be able to generate these.
Let $\alpha$ be a permutation of $\ZZd$ that exchanges $E$ and $H$,
two sets of cardinality four. Then $C_\alpha^{-1} P_\rho C_\alpha
\in A_2$, which implies that $P_\rho$ is generated by $A_1\cup A_2\cup
A_3$.
\qed
\end{proof}

\end{document}